\documentclass[10pt, a4paper, reqno]{amsart} 
\usepackage{amssymb}
\usepackage{amsmath}
\usepackage{amsthm}
\usepackage{thmtools}  
\usepackage[backref=page,colorlinks=false]{hyperref} 
\usepackage[capitalise]{cleveref} 
\usepackage{graphicx}
\usepackage{psfrag}
\usepackage[small]{caption}

\makeatletter

\def\myMRbibitem{\@ifnextchar[\my@lbibitem\my@bibitem}

\def\mybiblabel#1#2{\@biblabel{{\hyperref{http://www.ams.org/mathscinet-getitem?mr=#1}{}{}{#2}}}}

\def\myhyperanchor#1{\Hy@raisedlink{\hyper@anchorstart{cite.#1}\hyper@anchorend}}

\def\my@lbibitem[#1]#2#3#4\par{%
    \item[\mybiblabel{#2}{#1}\myhyperanchor{#3}\hfill]#4%
    \@ifundefined{ifbackrefparscan}{}{\BR@backref{#3}}%
    \if@filesw{\let\protect\noexpand\immediate
       \write\@auxout{\string\bibcite{#3}{#1}}}\fi\ignorespaces%
}

\def\my@bibitem#1#2#3\par{%
    \refstepcounter\@listctr
    \item[\mybiblabel{#1}{\the\value\@listctr}\myhyperanchor{#2}\hfill]#3%
    \@ifundefined{ifbackrefparscan}{}{\BR@backref{#2}}%
    \if@filesw\immediate\write\@auxout
        {\string\bibcite{#2}{\the\value\@listctr}}\fi\ignorespaces%
}

\makeatother

\declaretheoremstyle[
headfont=\normalfont\itshape,
bodyfont=\normalfont,
qed=\ensuremath{\triangleleft}
]{myremark}

\declaretheorem[numberwithin=section]{theorem}
\declaretheorem[sibling=theorem, name=Main Theorem]{maintheorem}
\declaretheorem[sibling=theorem]{lemma}
\declaretheorem[sibling=theorem, name=Main Lemma]{mainlemma}

\declaretheorem[sibling=theorem]{proposition}
\declaretheorem[sibling=theorem, style=remark]{claim}
\declaretheorem[sibling=theorem, style=myremark]{remark}

\declaretheorem[numbered=no, style=remark, name=Acknowledgement]{ack}

\crefname{claim}{Claim}{Claims} 
\crefname{mainlemma}{Main Lemma}{??} 

\crefname{section}{\S}{\S\S} 
\Crefname{section}{\S}{\S\S} 
\crefname{subsection}{\S}{\S\S} 
\Crefname{subsection}{\S}{\S\S}

\renewcommand*{\backref}[1]{}
\renewcommand*{\backrefalt}[4]{\quad \tiny 
    \ifcase #1 (Not cited.)%
    \or        (Cited on page~#2.)%
    \else      (Cited on pages~#2.)%
    \fi}

\hypersetup{bookmarksdepth = 3} 
\numberwithin{equation}{section}         

\newcommand{\R}{\mathbb{R}}

\newcommand{\N}{\mathbb{N}}
\newcommand{\E}{\mathbb{E}}
\newcommand{\F}{\mathbb{F}}
\newcommand{\cB}{\mathcal{B}}
\newcommand{\cD}{\mathcal{D}}

\newcommand{\cM}{\mathcal{M}}
\newcommand{\cR}{\mathcal{R}}

\renewcommand{\epsilon}{\varepsilon}
\renewcommand{\phi}{\varphi}
\renewcommand{\emptyset}{\varnothing}
\renewcommand{\setminus}{\smallsetminus}
\newcommand{\SL}{\mathrm{SL}}
\newcommand{\GL}{\mathrm{GL}}

\newcommand{\Aut}{\mathrm{Aut}}
\newcommand{\m}{\mathfrak{m}}
\newcommand{\s}{\mathfrak{s}}

\newcommand{\wed}{{\mathord{\wedge}}} 
\renewcommand{\angle}{\measuredangle}
\DeclareMathOperator{\interior}{int}
\DeclareMathOperator{\diam}{diam}

\newcommand{\arxiv}[1]{Preprint \href{http://arxiv.org/abs/#1}{arXiv:{#1}}}

\begin{document}

\title{Generic Linear Cocycles over a Minimal Base}
\author{Jairo Bochi}
\thanks{Partially supported by CNPq and FAPERJ}
\address{Pontif\'icia Universidade Cat\'olica do Rio de Janeiro (PUC--Rio)}
\urladdr{www.mat.puc-rio.br/$\sim$jairo}
\email{jairo@mat.puc-rio.br}
\date{February, 2013} 



\maketitle

\begin{abstract}
We prove that a generic linear cocycle over a minimal base dynamics of finite dimension has the property that the Oseledets splitting with respect to any invariant probability coincides almost everywhere with the finest dominated splitting. Therefore the restriction of the generic cocycle to a subbundle of the finest dominated splitting is uniformly subexponentially quasiconformal. This extends a previous result for $\SL(2,\R)$-cocycles due to Avila and the author.
\end{abstract}

\section{Introduction}

\subsection{Statement of the result}
Let $X$ be a compact Hausdorff space,
and let $\E$ be a real vector bundle with base space $X$.
We will always assume that fibers $\E(x)$ have constant finite dimension.

Let $T \colon X \to X$ be an homeomorphism.
A vector bundle automorphism covering $T$
is a map $A \colon \E \to \E$ whose restriction
to an arbitrary fiber $\E(x)$ 
is a linear isomorphism onto the fiber $\E(Tx)$;
this isomorphism will be denoted by $A(x)$.
Let $\Aut(\E,T)$ the set of these automorphisms.
When the vector bundle is trivial, an automorphism is usually called a \emph{linear cocycle}.

We endow $\E$ with a Riemannian metric,
and $\Aut(\E,T)$ with the uniform topology, that is, the topology induced by the distance 
$$
d(A,B) := \sup_{x\in X} \|A(x) - B(x)\| \, ,
$$
where $\| \mathord{\cdot} \|$ denotes the operator norm induced by the Riemannian metric.

\medskip

Given $A \in \Aut(\E,T)$, an (ordered) splitting of the vector bundle
$$ 
\E = \E_1 \oplus \E_2 \oplus \cdots \oplus \E_k
$$ 
is called \emph{dominated} (also \emph{exponentially separated})
if it is $A$-invariant and
there are constants $c>0$ and $\tau > 1$ such that
for all $x\in X$ and 
all unit vectors $v_1 \in \E_1(x)$, \dots, $v_k \in \E_k(x)$,
we have
$$
\frac{\|A^n(x) \cdot v_i\|}{\|A^n(x) \cdot v_{i+1}\|} > c \tau^n \, ,
\quad \forall n \ge 0.
$$
(In fact, it is always possible to choose an \emph{adapted} Riemannian metric so that $c=1$;
see \cite{Gourmelon}.)

There exists an unique such splitting into a maximal number $k$ of bundles, 
which is called the \emph{finest dominated splitting} of $A$. 
If $k=1$, this is just a \emph{trivial} splitting.
The finest dominated splitting refines any other dominated splitting of $A$.
(See e.g.\ \cite{BDV} for these and other properties of dominated spittings.)

\medskip

Given $A \in \Aut(\E,T)$, Oseledets theorem (see e.g.\ \cite{LArnold})
provides a set $R \subset X$ of full probability 
(i.e., such that $\mu(R)=1$ for every $T$-invariant probability measure $\mu$)
such that each fiber $\E(x)$ over a point $x\in R$
splits into subspaces having the same Lyapunov exponents.
This \emph{Oseledets splitting} is $A$-invariant, measurable, but in general not continuous.
For example, the dimensions of the subbundles may depend on the basepoint.
Notice that 
the Oseledets splitting always refines the finest dominated splitting,
since domination forces a gap between Lyapunov exponents.

\medskip

It is shown in \cite{BV} that for any ergodic measure $\mu$,
the generic automorphism $A$ has the property that the 
Oseledets splitting coincides $\mu$-almost everywhere
with the finest dominated splitting above the support of the measure.
In this paper we obtain this property simultaneously for all measures,
under suitable assumptions:

We say the space $X$ has \emph{finite dimension} if
it is homeomorphic to a subset of some euclidean space.
For instance, subsets of manifolds (assumed as usual to be Hausdorff and
second countable) have finite dimension.
We say that the homeomorphism $T$ is \emph{minimal} if every orbit is dense.

\begin{maintheorem}\label{t.main}
Let $T \colon X \to X$ be a minimal homeomorphism
of a compact space $X$ of finite dimension,
and let $\E$ be a vector bundle over $X$.
Let $\cR$ be the set of $A \in \Aut(\E,T)$ 
with the following property: 
for every $T$-invariant probability measure $\mu$, the Oseledets splitting with respect to $\mu$ 
coincides $\mu$-almost everywhere with the finest dominated splitting of $A$.
Then $\cR$ is a residual subset of $\Aut(\E,T)$.
\end{maintheorem}

Thus if $A \in \cR$ has a finest dominated splitting into $k$ subbundles 
then at almost every point $x$ with respect to each 
invariant probability measure, there are exactly $k$ different Lyapunov exponents at $x$.
Of course, these values are a.e.\ constant if the measure is ergodic;
they may however depend on the measure.

Since a minimal homeomorphism may have uncountably many ergodic measures, 
Theorem~\ref{t.main} is not a consequence of the aforementioned result of \cite{BV}.
Actually, the theorem was proved first in the case of $\SL(2,\R)$-cocycles in \cite{AB}.

\medskip

It is evident that the minimality assumption is necessary for the validity of Theorem~\ref{t.main};
it is easy to see that it cannot be replaced e.g.\ by transitivity.
An example from \cite{AB} shows that it is not sufficient to assume that $T$ has a unique minimal set.
As in \cite{AB}, we do not know whether the assumption that $X$ has finite dimension is
actually necessary.

\subsection{Uniform properties}

An immediate consequence of the \cref{t.main} is that for the generic automorphism,
the Oseledets splitting varies continuously.
Another consequence is that the time needed to see a definite separation between 
expansion rates along different Oseledets subbundles is uniform.
All these properties are much stronger than those provided by the Oseledets theorem itself.
Let us discuss another uniform property that follows from Theorem~\ref{t.main},
and that depends on information on all invariant measures.

\medskip

If $L$ is a linear automorphism between inner product vector spaces, 
define the \emph{mininorm} of $L$ as 
$\m(L):= \|L^{-1}\|^{-1}$,
and the \emph{quasiconformal distortion} of $L$ as 
\begin{equation}\label{e.def_kappa}
\kappa(L) := \log \left( \frac{\| L \|}{\m(L)} \right) \, .
\end{equation}
For an interpretation of this quantity in terms of angle distortion, see \cite[Lemma~2.7]{BV}.

Let us say that an automorphism $A \in \Aut(\E,T)$ is 
\emph{uniformly subexponentially quasiconformal} 
if  for every $\epsilon>0$ there exists $c_\epsilon > 0$ such that 
$$
\kappa \big(A^n(x)\big) \le c_\epsilon + \epsilon n
\quad \text{for all $x\in X$, $n \ge 0$.}
$$

\medskip

Then, as an addendum to the \namecref{t.main}, we have:

\begin{proposition}\label{p.addendum}
The elements of $\cR$ are exactly the automorphisms $A \in \Aut(\E,T)$
whose restrictions $A|\E_i$ to the each bundle of the finest dominated splitting 
$\E_1 \oplus \cdots \oplus E_k$
are uniformly subexponentially quasiconformal.
\end{proposition}

\subsection{Applications}

It is shown in \cite{BN} that if $A\in \Aut(\E,T)$  is 
uniformly subexponentially quasiconformal then for every $\epsilon>0$, 
there is a Riemannian metric on $\E$ with respect to which 
the quasiconformal distortion is less than $\epsilon$;
moreover if $\epsilon$ is small then a perturbation of $A$ is conformal 
with respect to this metric.
Putting these results together with \cref{t.main}, it is possible to show the following:

\begin{theorem}[{\cite[Thrm.~2.3]{BN}}]\label{t.dense_conf}
Let $T \colon X \to X$ be a minimal homeomorphism
of a compact space $X$ of finite dimension, 
and let $\E$ be a vector bundle over $X$.
Then there exists a dense subset $\cD \subset \Aut(\E,T)$ with the following properties:
For every $A \in \cD$ 
there exists a Riemannian metric on the vector bundle $\E$
with respect to which the subbundles of the finest dominated splitting of $A$ are orthogonal,
and the restriction of $A$ to each of these subbundles is conformal.
Moreover, this metric is adapted in the sense of \cite{Gourmelon}.
\end{theorem}

This result should be useful to study the following question: 
\emph{When can an automorphism $A \in \Aut(\E,T)$ be approximated by another with a 
nontrivial dominated splitting?}

\subsection{Comments on the proof and organization of the paper}

To prove Theorem~\ref{t.main} we used ideas and tools developed in \cite{AB} 
to deal with the $\SL(2,\R)$ case.
The basic strategy for mixing different expansion rates on higher dimensions
is similar to that from \cite{BV}, but using a characterization of domination from \cite{BG} 
to find the suitable places to perturb.
As in \cite{BV}, the desired residual set is obtained as the 
set of continuity points of some semicontinuous function.

Despite these overlaps, dealing simultaneously with several Lyapunov exponents 
with respect to all invariant measures presented substantial new difficulties.
We introduce an especially convenient semicontinuous function $Z$ to measure quasiconformal distortion.
This function was in fact suggested by some ideas from \cite{BoBo}.
The proof that the mixing mechanism actually produces a discontinuity of $Z$ 
is also more delicate: it is essential not to be too greedy,
and instead attack only the points on $X$ where the distortion is comparatively large.
This is explained in \cref{ss.proof_main_lemma}.

\medskip

The paper is organized as follows:

In \cref{s.initial} we explain several preliminaries, and 
reduce the proof of \namecref{t.main} to a result (\cref{l.discontinuity})
on the existence of discontinuities of a certain function (related to $Z$).

In \cref{s.segment} we prove \cref{l.main},
which produces the suitable perturbations along a segment of orbit.

In \cref{s.global} we explain how to patch those local perturbations
to prove \cref{l.discontinuity} and therefore conclude.

\section{Initial considerations}\label{s.initial}

In this section, $X$ is a compact Hausdorff space $X$,
the map $T \colon X \to X$ is at least continuous,
and $\E$ is a vector bundle over $X$ of dimension $d$.

We denote the set of all $T$-invariant probability measures by $\cM(T)$.
A Borel set $B \subset X$ is said to have \emph{zero probability} (resp.\ \emph{full probability})
with respect to a continuous map $T \colon X \to X$ 
if $\mu(B)$ is $0$ (resp.\ $1$) for every $T$-invariant probability measure $\mu$.

\subsection{Semi-uniform subadditive ergodic theorem}
\Cref{p.addendum} is an equivalence between a uniform property on $\cM(T)$ and a uniform property on $X$.
The following \cref{t.susaet} is often useful to obtain equivalences of this kind.


Recall that a sequence of $f_n \colon X \to \R$ is called \emph{subadditive} if 
$f_{n+m} \le f_n + f_m \circ T^n$.

\begin{theorem}[Semi-uniform subadditive ergodic theorem; {\cite[Thrm.~1]{Schreiber}}, {\cite[Thrm.~1.7]{SS}}]\label{t.susaet}
Let $T \colon X \to X$ be a continuous map of a compact Hausdorff space $X$.
Given a subadditive sequence of continuous functions $f_n \colon X \to \R$, we have
$$
\sup_{\mu \in \cM(T)} \lim_{n \to \infty} \frac{1}{n} \int_X f_n \, d\mu = 
\lim_{n \to \infty} \frac{1}{n} \sup_{x \in X} f_n(x) \, .
$$
\end{theorem}
Notice that, by Fekete's lemma both limits above can be replaced by $\inf$'s.
Also recall that for every $\mu \in \cM(T)$,
by Kingman's subadditive ergodic theorem 
the sequence $f_n(x)/n$ actually converges to a value in $[-\infty,+\infty)$
for every point $x$ on a full probability subset.

\subsection{Maximal asymptotic distortion}\label{ss.K}

Recall the definition \eqref{e.def_kappa} of the quasiconformal distortion $\kappa$.
Notice that $\kappa$ is \emph{subadditive}, i.e., if
$L_i \colon E_i \to E_{i+1}$ ($i=1$, $2$)
are isomorphisms between inner product spaces, then 
$\kappa (L_2 L_1) \le \kappa(L_2) + \kappa(L_1)$.

\medskip

Given an automorphism $A \in \Aut(\E,T)$, define 
\begin{equation}\label{e.def_K}
K(A) := \inf_{n \ge 1} \frac{1}{n} \sup_{x\in X} \kappa(A^n(x)) \, .
\end{equation}
(By Fekete's lemma, the $\inf$ can be replaced by a limit.)
Being an infimum of continuous functions,
$K \colon \Aut(A,\E) \to [0,\infty)$ is upper semicontinuous.

Notice that $A$ is uniformly subexponentially quasiconformal (as defined in the Introduction)
if and only if $K(A) = 0$.

\medskip

If $L$ is an isomorphism between inner product vector spaces of dimension $d$,
its singular values (i.e., the eigenvalues of $(L^* L)^{1/2}$)
will be written as $\s_1(L) \ge \dots \ge \s_d(L)$;
so $\s_1(L) = \|L\|$ and $\s_d(L) = \m(L)$.

Given $A \in \Aut(\E,T)$,
the following \emph{Lyapunov exponents} exist for every $x$ in a full probability subset of $X$:
$$
\chi_i (A,x) := \lim_{n \to +\infty} \frac{1}{n} \log \s_i(A^n(x)) \, ,
\quad (i=1,\dots,d).
$$
Let us denote their averages with respect to some $\mu \in \cM(T)$ as:
$$
\chi_i(A,\mu) := \int_X \chi_i(A,x) \, d\mu(x).
$$
It follows from \cref{t.susaet} that:
\begin{equation}\label{e.K_and_Lyap}
K(A) = \sup_{\mu \in \cM(T)} \big[ \chi_1(A,\mu) - \chi_d(A,\mu) \big] \, .
\end{equation}
In particular, $A$ is uniformly subexponentially quasiconformal
if and only if for every point $x$ in a full probability subset,
all Lyapunov exponents of $A$ at $x$ are equal.

\subsection{Distortion inside the bundles of a dominated splitting}\label{ss.K_fine}

Let us review the basic robustness property of dominated splittings:

\begin{proposition}\label{p.DS_continuation}
Suppose that the automorphism $A \in \Aut(\E,T)$ has a dominated splitting 
$\E_1 \oplus \cdots \oplus \E_k$.
Then every automorphism $\tilde A$ sufficiently close to $A$ has a dominated splitting
$\tilde \E_1 \oplus \cdots \oplus \tilde \E_k$
such that, for each $i=1,\dots,k$, the fibers of $\tilde \E_i$ have the same dimension 
and are uniformly close to the fibers of $\E_i$. 
\end{proposition}

We call $\tilde \E_1 \oplus \cdots \oplus \tilde \E_k$ the \emph{continuation} of 
the originally given dominated splitting for $A$.
We remark that the continuation of a finest dominated splitting is not necessarily finest.

\medskip

For any $A \in \Aut(\E,T)$, define
$$
K_\mathrm{fine}(A) := \max_i K(A|\E_i) \, , 
$$
where $\E_1  \oplus \cdots \oplus \E_k$ is the finest dominated splitting of $A$.

Notice that if $A \in \Aut(\E,T)$ and $\F \subset \E$ is an $A$-invariant subbundle
then $K(A) \ge K(A | \F)$. In particular, we have:

\begin{proposition}\label{p.break_K} 
$K_\mathrm{fine}(A) \le K(A)$.
\end{proposition}

We use this to show the following:

\begin{proposition}\label{p.usc_K}
The map $K_\mathrm{fine} \colon \Aut(\E, T) \to [0,\infty)$ is upper semicontinuous.
\end{proposition}

\begin{proof}
Let $A \in \Aut(\E,T)$ have finest dominated splitting  $\E_1 \oplus \cdots \oplus \E_k$,
and let $\epsilon > 0$
Let $\tilde A$ be a perturbation of $A$,
and let $\tilde \E_1 \oplus \cdots \oplus \tilde \E_k$ be the continuation of the splitting,
as given by \cref{p.DS_continuation}.
Each restriction $\tilde A | \tilde \E_i$ is conjugated to a perturbation of $A | \E_i$.
Since $K$ is upper-semicontinuous and invariant under conjugation, 
we have $K(\tilde A | \tilde \E_i) \le K(A | \tilde \E_i) + \epsilon$.
Since the finest dominated splitting of $\tilde A$ refines $\tilde \E_1 \oplus \cdots \oplus \tilde \E_k$,
it follows from \cref{p.usc_K} that $K_\mathrm{fine} (\tilde A) \le {K_\mathrm{fine} (A) + \epsilon}$.
\end{proof}

Notice that the set $\cR$ from the statement of the \cref{t.main} 
(or from \cref{p.addendum}, which is now obvious)
is precisely $\{A \in \Aut(\E,T) ; \; K_\mathrm{fine} (A) = 0\}$,
which by the \lcnamecref{p.usc_K} above is a $G_\delta$ set.
The hard part of the proof of the \namecref{t.main} is to show that $\cR$ is dense.

Actually, we will see later that $\cR$ is the set of points of
continuity of $K_\mathrm{fine}$, and therefore it is a residual set.
However, it is not convenient to work with $K_\mathrm{fine}$ directly.
We will introduce alternative ways of measuring quasiconformal distortion
that will turn out to be more appropriate.

\subsection{Another measure of quasiconformal distortion}

Let $E$ and $F$ be inner product spaces of dimension $d$ and let $L \colon E \to F$ be an isomorphism.
Recall that $\s_1(L) \ge \dots \ge \s_d(L)$ denote the singular values of $L$.
Let $\lambda_i(L) := \log \s_i(L)$.
Define also 
$$
\sigma_0(L) := 0 \quad \text{and} \quad  
\sigma_i(L)  := \lambda_1(L) + \dots + \lambda_i(L) \text{ for } i=1,\dots,d.
$$
In particular, $\sigma_1(L) = \log \|L\|$ and $\sigma_d(L) = \log \left| \det L \right|$.

\medskip

Consider the graph of the function $i \in \{0,1,\dots, d\} \mapsto \sigma_i(L) \in \R$.
By affine interpolation we obtain a graph over the interval $[0,d]$,
which we call the \emph{$\sigma$-graph} of $L$. 
The fact that the sequence $\lambda_i(L)$ is non-increasing means that this graph is concave.
In particular, the $\sigma$-graph of $L$ is above the line joining $(0,0)$ and $(d,\sigma_d(L))$.
Let us define $\zeta(L)$ as the area between this line and the $\sigma$-graph (see \cref{fig.areas}).
This amounts to:
\begin{equation}\label{e.def_zeta}
\zeta(L) = \sigma_1(L) + \sigma_2(L) + \dots + \sigma_{d-1}(L) - \left( \frac{d-1}{2} \right) \sigma_d(L).
\end{equation}

\psfrag{s}[r][r]{$\sigma_i(L)$}
\psfrag{i}[t][t]{$i$}
\begin{figure}[hbt]
	\includegraphics[width=.45\textwidth]{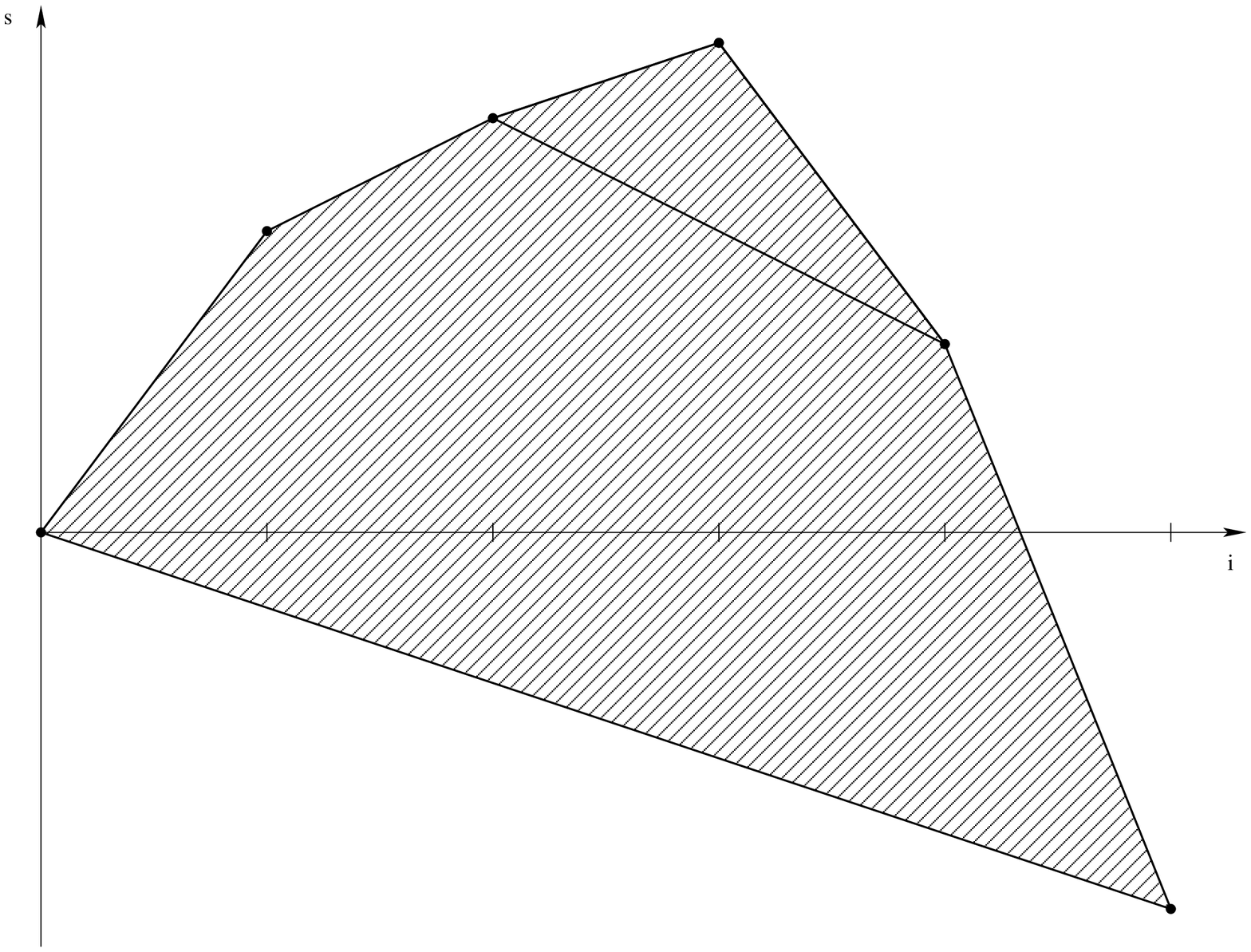}
	\caption{The upper curve is the $\sigma$-graph of some $L$. The shaded area is $\zeta(L)$. The area of the marked triangle is $\gamma_3(L)$.}
	\label{fig.areas}
\end{figure}

Of course, $\zeta(L) \ge 0$, and equality holds if and only 
if all singular values of $L$ are equal, i.e., $L$ is conformal.
(Actually, it is not difficult to show that for every fixed dimension $d$,
each quantity $\kappa$ and $\zeta$ is bounded by an uniform multiple of the other.)

\medskip

Like $\kappa$, the functions we have just defined enjoy the property of subadditivity:

\begin{proposition}\label{l.subadd}
The functions $\sigma_1$, \dots, $\sigma_{d-1}$ and $\zeta$ 
are subadditive and the function $\sigma_d$ is additive.
\end{proposition}

\begin{proof}
We recall some facts about exterior powers (see e.g.\ \cite[\S~3.2.3]{LArnold}).
Let $\wed^i E$ denote the $i$-th exterior power of $E$.
The inner product in $E$ induces an inner product on $\wed^i E$;
actually if $\{e_1, \dots, e_d\}$ is an orthonormal basis of $E$
then $\{e_{j_1} \wedge \dots \wedge e_{j_i}; \; 1\le j_1 < \dots < j_i \le d\}$ 
is an orthonormal basis of $\wed^i E$.
The isomorphism $L \colon E \to F$ induces an isomorphism
$\wed^i L \colon \wed^i E \to \wed^i F$,
and its norm is:
$$
\|\wed^i L\| =  \exp \sigma_i(L) \, .
$$
Since operator norms are submultiplicative,
it follows that $\sigma_i(\mathord{\cdot})$ is subadditive.
Moreover, since $\wed^d E$ is $1$-dimensional, $\sigma_d(\mathord{\cdot})$ is additive.
It follows from the definition \eqref{e.def_zeta}
that $\zeta(\mathord{\cdot})$ is subadditive.
\end{proof}
	
Let us introduce other quantities that will be used later,
namely the following ``half-gaps'' between the $\lambda$'s:   
$$
\gamma_i(L) := \frac{\lambda_i(L) - \lambda_{i+1}(L)}{2}  
= \frac{-\sigma_{i-1}(L) + 2\sigma_i(L) - \sigma_{i+1}(L) }{2}\, ,
\quad (i=1,\dots,d-1).
$$
Geometrically, these numbers are the areas of the triangles determined by three consecutive vertices 
in the $\sigma$-graph: see \cref{fig.areas}.
In particular, $\gamma_i(L) \le \zeta(L)$ for each~$i$.
On the other hand, the maximal half-gap is comparable to $\zeta(L)$,
as the following \lcnamecref{l.BoBo} shows:

\begin{lemma}\label{l.BoBo}
If $L$ is an isomorphism between inner product spaces of dimension $d \ge 2$ then
$$
\max_{i\in\{1,\dots,d-1\}} \gamma_i(L) \ge b_d \, \zeta(L) \, , 
$$
where $b_d \in (0,1]$ is a constant depending only on $d$.
\end{lemma}

\begin{proof}
A calculation shows that 
$\zeta(L) = \sum_{i=1}^{d-1} i(d-i) \gamma_i(L)$.
Therefore the \lcnamecref{l.BoBo} holds with 
\begin{equation*}
b_d := \left(\sum_{i=1}^{d-1} i(d-i) \right)^{-1} = \frac{6}{d (d^2 - 1)} \, . \qedhere
\end{equation*}
\end{proof}

Of course, \cref{l.BoBo} is just a property about concave graphs.
Despite its simplicity, this property will play a significant role here,
as is does (to a lesser extent) in \cite{BoBo}.

\subsection{Maximal quantities}

Given $A\in\Aut(\E,T)$, we define 
\begin{equation}\label{e.def_Z}
Z(A) := \inf_{n \ge 1} \frac{1}{n} \sup_{x\in X} \zeta(A^n(x)) \, .
\end{equation}
Then the function $Z \colon \Aut(\E,T) \to [0,\infty)$ is upper semicontinuous.

The analog of formula \eqref{e.K_and_Lyap} for $Z$ is: 
\begin{equation}\label{e.Z_and_Lyap}
Z(A) = \sup_{\mu \in \cM(T)} \zeta \Big( \mathrm{diag} \big(\chi_1(A,\mu), \chi_2(A,\mu), \dots, \chi_d(A,\mu) \big) \Big) \, .
\end{equation}

For any $A \in \Aut(\E,T)$, define
$$
K_\mathrm{fine}(A) := \max_i K(A|\E_i) \, , 
$$
where $\E_1  \oplus \cdots \oplus \E_k$ is the finest dominated splitting of $A$.

\begin{proposition}\label{p.break_Z} 
$Z_\mathrm{fine}(A) \le Z(A)$ for every $A\in \Aut(\E,T)$.
\end{proposition}

\begin{proof}
Let $A \in \Aut(\E, T)$,
and let $E_1 \oplus \dots \oplus E_k$ be the finest dominated splitting of $A$.
Take $i$ such that $K(A|\E_i) = K_\mathrm{fine}(A)$.
Let $m := \dim (\E_1 \oplus \cdots \oplus \E_{i-1})$ and $\ell := \dim \E_i$.
Applying \eqref{e.Z_and_Lyap} to the automorphism $A|\E_i$, we have
$$
Z(A|\E_i) = \sup_{\mu \in \cM(T)} \zeta \Big( \mathrm{diag} \big(\chi_{m+1}(A,\mu), \dots, \chi_{m+\ell}(A,\mu) \big) \Big) \, .
$$
It follows from the interpretation of $\zeta$ as an area that
$$
\zeta \Big( \mathrm{diag} \big(\chi_{m+1}(A,\mu), \dots, \chi_{m+\ell}(A,\mu) \big) \Big) 
\le 
\zeta \Big( \mathrm{diag} \big(\chi_{1}(A,\mu), \dots, \chi_{d}(A,\mu) \big) \Big) \, ,
$$
for every $\mu\in \cM(T)$.
Therefore $Z(A|\E_i) \le Z(A)$, as we wanted to show.
\end{proof}

Using \cref{p.break_Z} instead \cref{p.break_K},
the same argument that proved \cref{p.usc_K} yields:

\begin{proposition}\label{p.usc_Z}
The map $Z_\mathrm{fine} \colon \Aut(\E, T) \to [0,\infty)$ is upper semicontinuous.
\end{proposition}

Of course, $Z$ (resp.\ $Z_\mathrm{fine}$) vanishes 
if and only if $K$ (resp.\ $K_\mathrm{fine}$)
vanishes. 
Actually the main conclusions of \cref{ss.K,ss.K_fine}
could have been obtained using the functions $Z$ and $Z_\mathrm{fine}$ instead;
but we have preferred the proofs that seemed more natural.

\subsection{Setting up the proof}\label{ss.setup}

In the next sections, we will prove the following:

\begin{lemma}\label{l.discontinuity}
Let $T$ be a minimal homeomorphism of a space of finite dimension.
Then for every $\epsilon > 0$
there exists $\tilde A \in \Aut(\E,T)$ such that
$\| \tilde A (x) - A(x) \| < \epsilon$ for each $x\in X$ and 
$$
Z_\mathrm{fine}(\tilde A) < a_d Z_\mathrm{fine}(A) + \epsilon \, ,
$$
where $a_d \in (0,1)$ is a constant depending only on the dimension $d$.
\end{lemma}

An immediate consequence of \cref{l.discontinuity}
is that $A$ is a point of continuity of the function $Z_\mathrm{fine}(\mathord{\cdot})$
if and only if $Z_\mathrm{fine}(A) = 0$.
Since the points of continuity of a semicontinuous function on a Baire space form a residual set, 
the \cref{t.main} follows.

Therefore we are reduced to proving \cref{l.discontinuity}.
Actually, if suffices to prove it in the particular case that $A$ has no nontrivial dominated splitting:

\begin{proof}[Proof of the general case assuming the particular case]
Assume that \cref{l.discontinuity} is already proved 
for automorphisms of bundles of any dimension without nontrivial dominated splittings, 
thus providing a sequence $(a_d)$. 
Replacing each $a_d$ with $\max(a_1,\dots,a_d)$, we can assume that this sequence is nondecreasing.
	
Let $A \in \Aut(\E, T)$,
and let $\E_1 \oplus \dots \oplus \E_k$ be the finest dominated splitting of $A$.
Let $\epsilon > 0$, and take a positive $\epsilon' \ll \epsilon$.
Each restriction $A | \E_i$ is an automorphism with no dominated splitting and therefore,
by the particular case, there exists an $\epsilon'$-perturbation $B_i \in \Aut(\E_i,T)$ such that 
$Z(B_i) < a_d Z(A | \E_i) + \epsilon'$.
Let $\tilde A \in \Aut(\E,T)$ be such that 
$\tilde A | \E^i = B_i$; then $\tilde A$ is $\epsilon$-close to $A$. 
The finest dominated splitting of $\tilde A$ refines $\E_1 \oplus \dots \oplus \E_k$,
and thus by \cref{p.break_Z}, 
\begin{equation*}
Z_\mathrm{fine}(\tilde A) \le \max_i Z (\tilde A | \E_i) 
\le  \max_i \big( a_d Z (\E_i) + \epsilon \big) = 
a_d Z_\mathrm{fine}(A) + \epsilon \, . \qedhere
\end{equation*}
\end{proof}

\begin{remark}
The validity of \cref{l.discontinuity} is equivalent to the validity
of an analog statement for $K_\mathrm{fine}$.
The reason why $Z_\mathrm{fine}$ is more convenient to work with is that
we know how to prove (the particular case of) \cref{l.discontinuity} with a single perturbation,
while producing a discontinuity of $K_\mathrm{fine}$ 
would probably require a more complicated procedure.
\end{remark}

\begin{remark}\label{r.why_not}
Other upper semicontinuous functions on $\Aut(\E,T)$ that suggest themselves are:
$$
\Sigma_i (A) := \inf_{n \ge 1} \frac{1}{n} \sup_{x \in X} \sigma_i(A^n(x)) \, ,
\quad i=1,\dots,d.
$$
At first sight, these may seem the ``right'' functions to consider,
especially since the proof from \cite{BV} consists in finding 
a discontinuity of an analogue function 
(where the $\sup$ is replaced by an integral).
However, it is not clear how to actually use these functions to prove the \cref{t.main}.
\end{remark}

\section{Reducing non-conformality along segments of orbit} \label{s.segment}

This section is devoted to the proof of the following result, 
which plays a role similar to Lemma~2 in \cite{AB}:

\begin{mainlemma}\label{l.main}
Suppose that $T$ is minimal and without periodic orbits, 
$A \in \Aut(\E,T)$ has no nontrivial dominated splitting,
and $\epsilon>0$.
Then there exists $N \in \N$ with the following properties:
For every $x\in X$ and every $n \ge N$, there exist a sequence of linear maps
$$
\E(x) \xrightarrow{L_0} \E(Tx) \xrightarrow{L_1} \E(T^2 x) \xrightarrow{L_2} \cdots  
\xrightarrow{L_{n-1}} \E(T^n x)
$$
with $\|L_j - A(T^j(x))\| < \epsilon$ for each $j$
and such that 
$$
\frac{1}{n}\zeta \big( L_{n-1} \dots L_0 \big) < a_d \, Z(A) + \epsilon \, .
$$
where $a_d \in (0,1)$ is a constant depending only on the dimension $d$.
\end{mainlemma}

\subsection{Preliminary lemmas}

If $\E_1 \oplus \dots \oplus E_k$ is a nontrivial dominated splitting for some $A \in \Aut(\E,T)$,
then its \emph{indices} are the numbers:
$$
\dim (\E_1) , \ \dim (\E_1 \oplus \E_2), \  \dots, \ \dim (\E_1 \oplus \dots \oplus \E_{k-1}).   
$$
We will need the following implicit characterization of these indices: 

\begin{theorem}[{\cite[Thrm.~A]{BG}}]\label{t.BG}
An automorphism $A \in \Aut(\E,T)$
has a dominated splitting of index $i$
if and only if 
there exist $c>0$, $\tau>1$ such that
$$
\frac{\s_i(A^n(x))}{\s_{i+1}(A^n(x))} > c \tau^n  \quad \text{for all $x \in X$ and $n \ge 0$.}
$$
\end{theorem}

In other words, the indices of domination correspond to exponentially large gaps
between the singular values.

\medskip

Absence of domination permits us to significantly change the orbits of vectors 
by performing small perturbations.
One operation of this kind is described by the following \lcnamecref{l.moving spaces}:

\begin{lemma}\label{l.moving spaces}
Assume that $A \in \Aut (\E, T)$ has no dominated splitting of index $i$.
Then for every $\epsilon>0$ there exist $m \in \N$ and a nonempty open set $W \subset X$
with the following properties:
For every $x \in W$ 
and every pair of subspaces $E \subset \E(x)$, $F \subset \E(T^m x)$ with respective dimensions $i$ and $d-i$,
there exist a sequence of linear maps
$$
\E(x) \xrightarrow{L_0} \E(Tx) \xrightarrow{L_1} \E(T^2 x) \xrightarrow{L_2} \cdots  
\xrightarrow{L_{m-1}} \E(T^m x)
$$
with $\|L_j - A(T^j x)\| < \epsilon$ 
for each $j$
and such that 
$$
L_{m-1} \cdots L_0 (E) \cap F \neq \{0\}.
$$
\end{lemma}

For the proof, we will need the following standard result,
which can be shown by the same arguments as in the proof of \cite[Prop.~7.1]{BV}.

\begin{lemma}\label{l.one half}
For any $C>0$ and any $\alpha>0$, there exists $m \in \N$ with the following properties.
If $L_0$, $L_1$, \ldots, $L_{k-1} \in \GL(d,\R)$
satisfy $\|L_k^{\pm 1}\| \le C$,
and $v$, $w \in \R^d$ are non-zero vectors such that
$$
\frac{\|L_{k-1} \cdots L_0  w\| \, / \, \|w\|}{\|L_{k-1} \cdots L_0  v\| \, / \, \|v\|} 
> \frac{1}{2} \, ,
$$
then there exist non-zero vectors $u_0$, $u_1$, \ldots, $u_k \in \R^d$
such that $u_0 = v$, $u_k = L_{k-1} \cdots L_0 (w)$, and 
$$
\angle \big( u_{j+1}, L_j(u_j) \big) < \alpha \quad \text{for each $j=0,\ldots,k-1$.}
$$
\end{lemma}

\begin{proof}[Proof of \cref{l.moving spaces}]
Suppose $A \in \Aut(\E,T)$ has no dominated splitting of index~$i$.
Let $\epsilon>0$ be given.
Let $C>1$ be such that $\|A(x)^{\pm 1}\| \le C$ for all $x\in X$.
Fix a positive $\alpha \ll \epsilon$,
and let $k = k(C,\alpha) \in \N$ be given by Lemma~\ref{l.one half}.
Define open sets
$$
W(m) := \left\{ x \in M ; \; \frac{\s_{i+1}(A^m(x))}{\s_i(A^m(x))} > C^{2k} (1/2)^{m/k - 1} \right\} \, .
$$
Notice that if $W(m) = \emptyset$ for every sufficiently large $m$ 
then by \cref{t.BG} there is a dominated splitting of index $i$,
contradicting the hypothesis.
Therefore we can fix $m>k$ such that $W = W(m) \neq \emptyset$.

Now fix a point $x \in W$ and spaces $E \subset \E(x)$, $F \subset \E(T^m x)$ with respective dimensions $i$ and $d-i$.
For simplicity, write $P = A^m(x)$.

\begin{claim}
There exist unit vectors $v \in E$ and $w \in P^{-1}(F)$ such that
$\|Pv\| \le \s_i(P)$ and $\|Pw\| \ge \s_{i+1}(P)$.
\end{claim}

\begin{proof}[Proof of the claim]
Let $\{e_1, \dots, e_d\}$ be a basis of $\E(x)$ formed
by eigenvectors of $(P^* P)^{1/2}$ corresponding to the eigenvalues $\s_1(P) \ge \dots \ge \s_d(P)$.
Let $\tilde E$ be the space spanned by $e_i$, \dots, $e_d$.
Since $\dim E = i$, the intersection $E \cap \tilde E$ contains a unit vector $v$.
Then $\|Pv\| \le \s_i(P)$, proving the first part of the claim.
The proof of the second part is analogous.
\end{proof}

\begin{claim}
There exists $\ell$ with $0 \le \ell < m-k$ such that
$$
\frac{\|A^{k+\ell}(x) \cdot w\| \, / \,  \|A^{\ell}(x) \cdot w\|}
{\|A^{k+\ell}(x) \cdot v\| \, / \, \|A^\ell(x) \cdot v\|}
> \frac{1}{2} \, .
$$
\end{claim}

\begin{proof}[Proof of the claim]
Assume the contrary.
It follows that:
$$
\frac{\s_{i+1}(P)}{\s_{i}(P)} \le \frac{\|Pw\|}{\|Pv\|} \le
\left( \frac{1}{2} \right)^{\lfloor m/k \rfloor} C^{2k} \, ,
$$
which contradicts the fact that $x\in W$.
\end{proof}

Next we apply Lemma~\ref{l.one half} to the vectors
$\tilde v = A^\ell(x) \cdot v$, $\tilde w = A^\ell(x) \cdot w$
and the linear maps
$\tilde L_0 = A(T^\ell x)$, \dots, $\tilde L_{k-1} = A(T^{\ell+k-1} x)$.
We obtain non-zero vectors $u_0$, \dots, $u_k$ such that
$u_0 = v$, $u_k =  A^{\ell+k}(x) \cdot w$, and 
$\angle (u_{j+1}, A(T^{\ell+j} x) \cdot u_j) < \alpha$ for each $j = 0, \dots, k-1$.

To conclude the proof, we need to define the linear maps $L_0$, \dots, $L_{m-1}$.
Since $\alpha$ is small, for each $j = 0, \dots, k-1$ 
we can find an $\epsilon$-perturbation $L_{\ell+j}$ of $A(T^{\ell+j} x)$
such that $L_j(u_j)$ and $u_{j+1}$ are collinear.
We define the remaining maps as:
$$
L_j = A(T^j x) \quad \text{if } 0 \le  j \le \ell \text{ or } \ell +k \le j \le m \, .
$$
Then $L_{m-1} \dots L_0 (v)$ is collinear to $A^m(w)$.
This proves \cref{l.moving spaces}.
\end{proof}

The next \lcnamecref{l.BV} indicates how the perturbations that 
\cref{l.moving spaces} provides can be used to manipulate singular values. 
For simplicity of notation, we state the \lcnamecref{l.BV}
in terms of matrices instead of bundle maps.

\begin{lemma}\label{l.BV}
Let $P$, $Q \in \GL(d,\R)$ and $i \in \{1, \dots, d-1\}$. 
Then there are subspaces $E$, $F \subset \R^d$ with respective dimensions $i$, $d-i$
and with the following property: 
If $R \in \GL(d,\R)$ satisfies 
$R(E) \cap F \neq \{0\}$ then
$$
\sigma_i (QRP) \le \sigma_i (P) + \sigma_i(Q) 
- 2 \min \big\{ \gamma_i(P), \gamma_i(Q)  \big\} 
+ c_d \max \{1, \log \|R\|\} \, , 
$$
where $c_d>0$ depends only on $d$.
\end{lemma}


A similar estimate appears in the proof of \cite[Prop~4.2]{BV}.

\begin{proof}
Let $P$, $Q$, and $i$ 
be given.
Fix an orthonormal basis $\{e_1, \dots, e_d\}$ of eigenvectors of $(P P^*)^{1/2}$ 
corresponding to the eigenvalues $\s_1(P)$, \dots, $\s_d(P)$,
and let $E$ be the subspace spanned $e_1$, \dots, $e_i$.
Analogously, fix an orthonormal basis $\{f_1, \dots, f_d\}$ of eigenvectors of $(Q^* Q)^{1/2}$ 
corresponding to the eigenvalues $\s_1(Q)$, \dots, $\s_d(Q)$,
and let $F$ be the subspace spanned by $f_{i+1}$, \dots, $f_d$.

Now take $R \in \GL(d,\R)$ such that 
$R(E) \cap F \neq \{0\}$.

Define also $\bar{e}_j := \s_j(P) \, P^{-1}(e_j)$ and $\bar{f}_j := (\s_j(Q))^{-1} \, Q(e_j)$,
for $j=1,\dots,d$.
Then $\{\bar{e}_1, \dots, \bar{e}_d\}$ and $\{\bar{f}_1, \dots, \bar{f}_d\}$ are orthonormal bases
formed by eigenvectors of $(P^* P)^{1/2}$ and $(Q Q^*)^{1/2}$, respectively.

As in the proof of \cref{l.subadd}, we will use exterior powers.
Consider the following subsets of $\wed^i \R^d$:
\begin{align*}
\cB_0 &= \big\{ \bar{e}_{j_1} \wedge \cdots \wedge \bar{e}_{j_i} ; \; 1 \le j_1 < \dots < j_i \le d \big\}, \\
\cB_1 &= \big\{ e_{j_1} \wedge \cdots \wedge e_{j_i} ; \; 1 \le j_1 < \dots < j_i \le d \big\}, \\
\cB_2 &= \big\{ f_{j_1} \wedge \cdots \wedge e_{j_i} ; \; 1 \le j_1 < \dots < j_i \le d \big\}, \\
\cB_3 &= \big\{ \bar{f}_{j_1} \wedge \cdots \wedge \bar{f}_{j_i} ; \; 1 \le j_1 < \dots < j_i \le d \big\},
\end{align*}
each of them endowed with the lexicographical order.
These are all orthonormal bases of $\wed^i \R^d$.
We represent the maps $\wed^i P$, $\wed^i R$, $\wed^i Q$ as $\binom{d}{i} \times \binom{d}{i}$
matrices $\mathbf{P}$, $\mathbf{R}$, $\mathbf{Q}$
with respect to these bases
$$
(\wed^i \R^d, \cB_0) \xrightarrow{\wed^i P} 
(\wed^i \R^d, \cB_1) \xrightarrow{\wed^i R} 
(\wed^i \R^d, \cB_2) \xrightarrow{\wed^i Q} 
(\wed^i \R^d, \cB_3) \, .
$$
Then the matrices $\mathbf{P}$ and $\mathbf{Q}$ are diagonal with positive diagonal entries.
The biggest and the second biggest entries of  $\mathbf{P}$ are respectively
$$
\mathbf{P}_{11} = \s_1(P) \dots \s_i(P) \quad \text{and} \quad
\mathbf{P}_{22} = \s_1(P) \dots \s_{i-1}(P) \s_{i+1}(P) \, .
$$
Analogously, the biggest and the second biggest entries of  $\mathbf{Q}$ are respectively
$$
\mathbf{Q}_{11} = \s_1(Q) \dots \s_i(Q) \quad \text{and} \quad
\mathbf{Q}_{22} = \s_1(Q) \dots \s_{i-1}(Q) \s_{i+1}(Q) \, .
$$

\begin{claim}\label{cl.zero_corner}
	$\mathbf{R}_{11} = 0$.
\end{claim}

\begin{proof}[Proof of the claim]
By assumption, there exist a non-zero vectors $w \in E \cap R^{-1}(F)$.
Choose $\ell \in \{1, \dots, i\}$ such that 
$\{e_1, \dots, e_{\ell-1}, w, e_{\ell+1}, \dots, e_i\}$ is a basis for $E$.
Therefore the first element of the basis $\cB_1$  
is a multiple of 
$\xi := e_1 \wedge \dots \wedge e_{\ell-1} \wedge w \wedge e_{\ell+1} \wedge \dots \wedge e_i$.
We have
$$
(\wed^i R)(\xi) = 
R(e_1) \wedge \dots \wedge R(e_{\ell-1}) \wedge R(w) \wedge R(e_{\ell+1}) \wedge \dots \wedge R(e_i) \, .
$$
Write each $R(e_j)$ as a linear combination of vectors $f_1$, \dots, $f_d$,
write $R(w)$ (which is in $F$) as a linear combination of vectors $f_{i+1}$, \dots, $f_d$, 
and substitute in the expression above.
We obtain a linear combination of vectors 
$f_{j_1} \wedge \cdots \wedge f_{j_i}$ where $f_1 \wedge \cdots \wedge f_i$
does not appear. 
This means that the first coordinate of $(\wed^i R)(\xi)$ with respect to the basis $\cB_2$ is zero.
Therefore $\mathbf{R}_{11}=0$.
\end{proof}

Now let $\mathbf{M} = \mathbf{Q}\mathbf{R}\mathbf{P}$, i.e., the matrix that represents $\wed^i(QRP)$
with respect to the bases $\cB_0$ and $\cB_3$.
Then the norm of $\mathbf{M}$ is $\exp \sigma_i(QRP)$.
This norm is comparable to $\max_{\alpha,\beta} |\mathbf{M}_{\alpha \beta}|$.
We estimate each entry as follows:
$$
|\mathbf{M}_{\alpha \beta}| 
=   \mathbf{Q}_{\alpha \alpha} \, |\mathbf{R}_{\alpha \beta}| \, \mathbf{P}_{\beta \beta} 
\le |\mathbf{R}_{\alpha \beta}| \, \max\{\mathbf{Q}_{11}\mathbf{P}_{22}, \mathbf{Q}_{22}\mathbf{P}_{11} \} \, .
$$
On one hand, $\max_{\alpha,\beta} |\mathbf{R}_{\alpha \beta}|$ is comparable to 
$\|\mathbf{R}\| = e^{\sigma_i(R)} \le \|R\|^i$.
On the other hand, 
$$
\log (\mathbf{Q}_{11}\mathbf{P}_{22}) = \sigma_i(P) + \sigma_i(Q) - 2\gamma_i(P), \quad
\log (\mathbf{Q}_{22}\mathbf{P}_{11}) = \sigma_i(P) + \sigma_i(Q) - 2\gamma_i(Q),
$$
and so the \lcnamecref{l.BV} follows.
\end{proof}

\subsection{Proof of the \cref{l.main}} \label{ss.proof_main_lemma}

First, let us give an outline of the proof.
If the segment of orbit $\{x, Tx, \dots, T^{n-1} x \}$ is long,
then by minimality it will regularly visit the sets from \cref{l.moving spaces}
where the lack of domination is manifest.
We will choose a single one of those visits,
and then perform a perturbation of the kind given by \cref{l.moving spaces} on a relatively short subsegment,
in order to obtain by \cref{l.BV} a drop in one $\sigma_i$ value of the long product.
We have to assure ourselves that this drop is a significant one.

Similar strategies are used in \cite{AB} and \cite{BV}.
In \cite{BV}, the short perturbative subsegment is chosen basically halfway along the segment;
that this is a suitable position for perturbation is a consequence of Oseledets theorem.
In the minimal $\SL(2,\R)$ situation considered in \cite{AB}, the middle position is not necessarily
the most convenient one, but nevertheless it is easy to see that there exists a suitable position that produces a big drop.

The considerations here are more delicate.
We actually apply \cref{l.moving spaces,l.BV} to the index $i_0$ which maximizes the half-gap 
$\gamma_{i_0}(A^n(x))$ and so is likely to produce a bigger drop in the $\zeta$-area (see \cref{fig.areas}).
Suppose we break $A^n(x) = Q P$ into left and right unperturbed subsegments (disregarding the short middle term). 
Similarly to \cite{AB}, we choose the breaking point so that $\gamma_{i_0}(P) \simeq \gamma_{i_0}(Q)$.
Then we need to estimate the drop in $\zeta$.
By subadditivity, $\sigma_i(A^n(x)) \le \sigma_i(P) + \sigma_i(Q)$ for each $i$. 
On the other hand, since the lengths $k$ and $n-k$ of $P$ and $Q$ are big, 
the values $k^{-1}\zeta(P)$ and $(n-k)^{-1}\zeta(Q)$ are essentially bounded by $Z(A)$. 
We can assume that for the point $x$ under consideration, 
the value $n^{-1}\zeta(A^n(x))$ is already sufficiently close to $Z(A)$, 
because otherwise no perturbation is needed. 
It follows that $\zeta(A^n(x)) \simeq \zeta(P) + \zeta(Q)$ and therefore
$\sigma_i(A^n(x)) \simeq \sigma_i(P) + \sigma_i(Q)$ for each $i$.
This allows us to recover an ``Oseledets-like'' situation and carry on the estimates easily.
The actual argument is more subtle, because in order to prove the \namecref{l.main} 
we need to consider points $x$ such that $n^{-1}\zeta(A^n(x))$ is close, but not extremely close, to $Z(A)$. 
We proceed with the formal proof.

\begin{proof}[Proof of the \namecref{l.main}]
Let $b = b_d$ be given by \cref{l.BoBo}, and define 
\begin{equation}\label{e.def_a}
a = a_d := \frac{1}{1+b/2} \, .
\end{equation}
Let $A \in \Aut(\E,T)$ be without nontrivial dominated splitting,
and let $\epsilon>0$.
Take a positive number $\delta \ll \epsilon$; how small it needs to be will become apparent along the proof.

For each $i=1,\dots,d-1$, 
we apply \cref{l.moving spaces} and thus obtain 
an integer $m_i$ and a nonempty open set $W_i \subset X$ with the following property:
along segments of orbits of length $m_i$ starting from $W_i$, 
we can $\epsilon$-perturb the linear maps in order to make 
any given $i$-dimensional space intersect any given $(d-i)$-dimensional space.

Since $T$ is minimal, there exists $m' \in \N$ such that
\begin{equation}\label{e.hitting_time}
\bigcup_{j=0}^{m'} T^j(W_i) = X \quad \text{for each $i=1,\dots,d-1$.}
\end{equation}
Let also $m'' \in \N$ be such that
\begin{equation}\label{e.zeta_control}
j \ge m''  \ \Rightarrow \  \zeta(A^j(y)) < \big( Z(A) + \delta\big) j , \ \forall y \in X \, .
\end{equation}

Take
\begin{equation}\label{e.N}
N \ge \delta^{-1} \max \{ m_1, \dots, m_{d-1}, m', m'' \}.
\end{equation}
Fix any point $x\in X$ and any $n \ge N$.
We can assume that 
\begin{equation}\label{e.not_so_small}
\frac{1}{n}\zeta (A^n(x)) \ge a \, Z(A) \, ,
\end{equation}
because otherwise the unperturbed maps $L_j = A(T^j(x))$ satisfy the conclusion of the \namecref{l.main}.

Let $i_0 \in \{1,\dots,d-1\}$ be such that
$\gamma_{i_0}(A^n(x)) = \max_{i} \gamma_i(A^n(x))$. 
Thus, by \cref{l.BoBo}, 
\begin{equation}\label{e.BoBo}
\gamma_{i_0}(A^n(x)) \ge b \, \zeta(A^n(x)) \, .
\end{equation}

Let us write $m_0 = m_{i_0}$, for simplicity.
Given an integer $k \in [0,n-m_0]$, we factorize $A^n(x)$ as $Q_k R_k P_k$, where
$$
P_k : = A^k(x) , \quad R_k : = A^{m_0}(T^k x),  \quad Q_k : = A^{n-k-m_0} (T^{k+m_0} x) \, ;
$$

In what follows, we will use big~O notation;
the comparison constants are allowed to depend only on $A$ (and $d$).
\begin{claim}
We can find $k \in [m'', n - m_0 - m'']$ such that $T^k x \in W_{i_0}$ and 
\begin{equation}\label{e.close_gammas}
\big| \gamma_{i_0}(P_k) - \gamma_{i_0}(Q_k) \big| \le O(\delta n) \, .
\end{equation}
\end{claim}

\begin{proof}[Proof of the claim]
Notice the following facts:
\begin{itemize}
	\item $\big| \gamma_{i_0} (A^{j+1}(x)) - \gamma_{i_0} (A^{j}(x)) \big| \le O(1)$ for every $j$.
	\item So, letting $\Delta_j := \gamma_{i_0}(A^j(x)) - \gamma_{i_0}(A^{n - j}(T^j x))$, we have 
	$| \Delta_{j+1} - \Delta_j| \le O(1)$.
	\item Since $\Delta_0 = -\Delta_n$, there exists $j_0 \in [0,n]$ such that $|\Delta_{j_0}| \le O(1)$.
	\item So there exists $j_1 \in [m'', n - m_0 - m'']$ such that $|\Delta_{j_1}| \le O(m'' + m_0)$.
	\item So, by \eqref{e.hitting_time}, there exists $k \in [m'', n - m_0 - m'']$ such that $T^k x \in W_{i_0}$ and $|\Delta_k| \le O(m'' + m_0 + m')$.
\end{itemize}
Since the right hand side of \eqref{e.close_gammas} is $\le |\Delta_k| + O(m_0)$,
the claim follows from \eqref{e.N}.
\end{proof}

Let $k$ be fixed from now on, and write $P=P_k$, $R=R_k$, $Q=Q_k$.

\medskip

Let $E \subset \E(T^k x)$ and $F \subset \E(T^{k+m_0} x)$ be the subspaces 
with respective dimensions $i_0$ and $d-i_0$
obtained by applying \cref{l.BV} to the maps $P$ and $Q$.
Since $T^k x \in W_{i_0}$, we can apply \cref{l.moving spaces}
and find linear maps $\tilde L_j \colon \E(T^{k+j} x) \to \E(T^{k+j+1} x)$ (where $j=0$,\dots,$m_0-1$)
each $\epsilon$-close 
to the respective $A(T^{k+j} x)$,
whose product $\tilde R := \tilde L_{m_0-1} \cdots \tilde L_0$ satisfies
$\tilde{R}(E) \cap F \neq \{0\}$.
The maps $L_j$ ($j=0$, \dots, $n-1$) that we are looking for are 
$L_j = \tilde L_{j-k}$ if $k \le j < k + m_0$,
and $L^j = A(T^j x)$ otherwise.
So their product is $L_{n-1} \cdots L_0 = Q \tilde R P$.
Notice that $\| \tilde R \| \le O(m_0) \le O(\delta n)$.
Therefore \cref{l.BV} gives:
\begin{equation}\label{e.BV_again}
\sigma_{i_0}(Q \tilde{R} P) 
\le \sigma_{i_0}(P) + \sigma_{i_0}(Q) 
- 2 \min \big\{ \gamma_{i_0}(P), \gamma_{i_0}(Q)  \big\} + O(\delta n).
\end{equation}

To conclude the proof, we need to estimate $\zeta(Q \tilde{R} P)$.
Begin by noticing that, as a consequence of \eqref{e.zeta_control},
\begin{equation}\label{e.sum_zetas}
\zeta(P) + \zeta(Q) \le Z(A) \, n + O(\delta n).
\end{equation}
Also, since $\sigma_i(R) \le O(m_0) \le O(\delta n)$, subadditivity and additivity give: 
\begin{equation}\label{e.perepeque} 
\sigma_i(P) + \sigma_i(Q)  
\begin{cases} 
	\ge \sigma_i(A^n(x)) - O(\delta n) &\quad \text{for each $i=1,\dots,d-1$,}  \\
    \le \sigma_d(A^n(x)) + O(\delta n) &\quad \text{for each $i=d$.}
\end{cases}
\end{equation}

\begin{claim}\label{cl.a_priori}
$\zeta(P) + \zeta(Q) - \zeta(A^n(x)) \ge - \gamma_{i_0} (P) - \gamma_{i_0}(Q) + \gamma_{i_0}(A^n(x)) 
- O(\delta n)$.
\end{claim}

\begin{remark}
Since \cref{cl.a_priori} is an important estimate in the proof, 
it is worthwhile to interpret it geometrically.
Consider the concave graphs of $\sigma_i(A^n(x))$ and $\sigma_i(P) + \sigma_i(Q)$.
By \eqref{e.perepeque}, modulo a small error, the first graph is below the second one
and their endpoints meet.
The quantities $\gamma_{i_0}(A^n(x))$ and $\gamma_{i_0} (P) + \gamma_{i_0}(Q)$ are the areas
of triangles touching the corresponding graphs, as in \cref{fig.areas}.
Now, if the first quantity is substantially bigger than the second quantity,
then concavity forces the existence of a large hole between the two graphs, 
and therefore the $\zeta$-area of the second graph is substantially bigger than
the $\zeta$-area of the first one.
\end{remark}

\begin{proof}[Proof of the claim]
Since the functions $\gamma_{i_0}$ and $\zeta$ are invariant under composition with homothecies,
we can assume for simplicity that $\sigma_d=0$, i.e., $\left| \det \right| = 1$,
for all the linear maps involved.
Notice that for any $L$ with $\left| \det L \right| = 1$, we have
$$
\zeta(L) + \gamma_{i_0}(L) = \sum_{i=1}^{d-1} u_i \, \sigma_i(L), \quad
\text{where } u_i := 
\begin{cases}
	1   &\text{if $|i-i_0|>1$,} \\
	1/2 &\text{if $|i-i_0|=1$,} \\
	2   &\text{if $i=i_0$.}
\end{cases}
$$
In particular,
\begin{multline*}
\zeta(P) + \gamma_{i_0}(P) + \zeta(Q) + \gamma_{i_0}(Q) - \zeta(A^n(x)) - \gamma_{i_0}(A^n(x)) 
\\ = \sum_{i=1}^{d-1} u_i \underbrace{\big[ \sigma_i(P) + \sigma_i(Q) - \sigma_i(A^n(x)) \big]}_{\ge - \delta n \ \text{(by \eqref{e.perepeque})}} 
\ge -d\delta n, 
\end{multline*}
which completes the proof of the claim.
\end{proof}

Next, we estimate
\begin{alignat*}{4}
\gamma_{i_0} (P) + \gamma_{i_0}(Q)
&\ge \gamma_{i_0}(A^n(x)) + \zeta(A^n(x)) &&- \zeta(P) - \zeta(Q) &&- O(\delta n) 
&\quad&\text{(by \cref{cl.a_priori})}\\
&\ge (b+1) \zeta(A^n(x))                  &&- \zeta(P) - \zeta(Q) &&- O(\delta n) 
&\ &\text{(by \eqref{e.BoBo})}\\
&\ge (b+1) a Z(A) \, n                    &&- Z(A) \, n           &&- O(\delta n) 
&\quad&\text{(by \eqref{e.sum_zetas})}\\
&= (ab+a-1) Z(A) \, n                   &&                      &&- O(\delta n) \, .
\end{alignat*}
Therefore, using \eqref{e.close_gammas}
\begin{align*}
2 \min\{\gamma_{i_0}(P),\gamma_{i_0}(Q)\} 
&=   \gamma_{i_0}(P) + \gamma_{i_0}(Q) - |\gamma_{i_0}(P)-\gamma_{i_0}(Q)| \\
&\ge  (ab+a-1) Z(A) \, n - O(\delta n)
\end{align*}
Substituting this into \eqref{e.BV_again} we obtain
$$
\sigma_{i_0}(Q \tilde{R} P) 
\le \sigma_{i_0}(P) + \sigma_{i_0}(Q) - (ab+a-1) Z(A) \, n + O(\delta n).
$$
So it follows from \eqref{e.perepeque} 
that
$$
\zeta (Q \tilde{R} P) 
\le \zeta(P) + \zeta(Q) - (ab+a-1) Z(A) \, n + O(\delta n).
$$
Using \eqref{e.sum_zetas} we obtain
$$
\zeta (Q \tilde{R} P) \le \underbrace{(2 - ab - a)}_{=a \text{ (by \eqref{e.def_a})}} Z(A) \, n
+ \underbrace{O(\delta n)}_{<\epsilon n} \, .
$$
This concludes the proof of the \namecref{l.main}.
\end{proof}

\section{Patching the perturbations} \label{s.global} 

Here we will use the \cref{l.main} to prove \cref{l.discontinuity}
and therefore the \namecref{t.main}.
The arguments are essentially the same as in \cite{AB}.

\medskip

To begin, we recall some results from \cite{AB} on zero probability sets.

\begin{theorem}[{\cite[Lemma~3]{AB}}]\label{t.boundary}
Let $X$ be a compact space of finite dimension,
and let $T \colon X \to X$ be a homeomorphism without periodic orbits. 
Then there exists a basis of the topology of $X$ consisting of sets $U$
such that $\partial U$ has zero probability.
\end{theorem}

This is the only place where we use the assumption that $X$ has finite dimension.
(Actually, the proof of the \lcnamecref{t.boundary} 
consists in finding sets $U$ such that no point in $X$
visits $\partial U$ more than $\dim X$ times.)

The next result follows from a simple Krylov--Bogoliubov argument:

\begin{lemma}[{\cite[Lemma~7]{AB}}]\label{l.freq}
Let $T \colon X \to X$ be a continuous mapping of a compact space $X$.
If $K \subset X$ is a compact set with zero probability then
for every $\epsilon>0$, there exists an open set $V \supset K$
and $n_* \in \N$ such that
$$
x \in X, \ n \ge n_*   \quad \Rightarrow \quad 
\#  \{x,Tx,\dots,T^{n-1} x\} \cap V < \epsilon n.
$$
\end{lemma}

We also need the following result that decomposes the space into two Rokhlin towers:

\begin{lemma}[{\cite[Lemma~6]{AB}}]\label{l.castle}
Let $X$ be a non-discrete compact space,
and let $T \colon X \to X$ be a minimal homeomorphism.
Then for any $N \in \N$, there exists an open set $B \subset X$ such that:
\begin{itemize}
\item the return time from $B$ to itself under iterations of $T$ assumes only the values $N$ and $N + 1$;
\item $\partial B$ has zero probability.
\end{itemize}
\end{lemma}

\medskip

Since we are working with non-necessarily trivial vector bundles $\E$,
we need to introduce local coordinates.

Let us fix a finite open cover $\{\hat{D}_m\}$ of $X$ by trivializing domains,
together with bundle charts $\xi_m \colon \hat{D}_m \times \R^d \to \E$.
For each $x \in \hat{D}_m$, the map $H_m(x) := \xi_m (x, \mathord{\cdot})$ is an isomorphism
from $\R^d$ to $\E(x)$.
We can assume that there is a finer cover $\{D_m\}$ of $X$
with $\overline{D_m} \subset \hat{D}_m$ for each $m$.

It is convenient to fix a constant $C>0$ such that:
\begin{equation}\label{e.H_bounds}
\big\| (H_m(x))^{\pm 1} \big\| \le C \quad\text{and}\quad
\zeta (H_m(x)) \le C, \quad \forall m, \ \forall x \in D_m \, . 
\end{equation}

Any $B \in \Aut(\E,X)$ can be represented in local coordinates 
by a family of (uniformly continuous) maps $B^{(m,m')} \colon X_m \cap T^{-1}(X_{m'}) \to \GL(d,\R)$ defined by:
\begin{equation}\label{e.local}
B^{(m,m')}(x) := \big( H_{m'} (Tx) \big)^{-1} \circ B(x) \circ H_m (x) \, , \quad x \in X_m \cap T^{-1}(X_{m'}) \, .
\end{equation}
Let us call this the \emph{$(m,m')$-local representation} of $B(x)$.

\medskip

Now we have all the tools we need to conclude the proof.

\begin{proof}[Proof of the \cref{l.discontinuity}] 
As explained in \cref{ss.setup}, it is sufficient to consider the particular case where
the automorphism $A \in \Aut(\E,T)$ has no nontrivial dominated splitting.
If the space $X$ is discrete then it consists of a single periodic orbit,
and it follows that $Z(A) = 0$.
So we can assume that $X$ is non-discrete, i.e., $T$ has no periodic orbits.

Fix $\epsilon>0$; we can assume that:
\begin{equation}\label{e.epsilon_1} \\
\epsilon < \inf_{x\in X} \m(A(x)) \, .
\end{equation}
As a consequence, if a linear map $L \colon \E(x) \to \E(Tx)$ is such that $\|L-A(x)\|<\epsilon$
then it is invertible; moreover $\zeta(L)$ is bounded by some $C_0 = C_0 (A,\epsilon)$.
Let $\epsilon' > 0$ be small enough so that:
\begin{align}
(1+C_0) \epsilon'   &< \epsilon/3 \, , \label{e.epsilon_3} \\
C^2(C^2+1)\epsilon' &< \epsilon   \, , \label{e.epsilon_2} 
\end{align}
where $C$ as in \eqref{e.H_bounds}.
Let $N  = N (A,\epsilon') \in \N$ be given by the \cref{l.main}.
We can assume that $N$ is large enough so that
\begin{equation}\label{e.boring}
\frac{2C}{N} < \frac{\epsilon}{3} \, .
\end{equation}
Recall that $\{D_m\}$ is a cover of $X$ by trivializing domains.
By uniform continuity of the local representations \eqref{e.local}, there exists $\rho>0$ such that 
\begin{equation}\label{e.unif_cont}
x, y \in D_m \cap T^{-1}(D_{m'}), \  d(x,y) < \rho \Rightarrow \ 
\big \|A^{(m,m')}(x) - A^{(m,m')}(y) \big\| < \epsilon' \, .
\end{equation}

\smallskip

Choose an open cover $\{W_i\}_{i=1,\dots,k}$ of $X$ with the following properties:
\begin{itemize}
	\item it refines the cover 
$\big\{D_{m_0} \cap T^{-1}(D_{m_1}) \cap \cdots \cap T^{-N-1}(D_{m_{N+1}}) \big\}_{m_0,\dots,m_{N+1}}$;
	\item $\diam T^j(W_i) < \rho$ for each $i=1, \dots, k$ and $j=0,1,\dots,N+1$;
	\item the sets $\partial W_i$ have zero probability.
\end{itemize}
(To guarantee the last requirement we use \cref{t.boundary}.)
For each $i=1, \dots, k$ and each $j=0,1,\dots,N+1$, we fix an index $m(i,j)$ such that
$T^j(W_i) \subset D_{m(i,j)}$.

Let $B$ be the set given by \cref{l.castle}.
Let $B_\ell$ be the set of points in $B$ whose first return to $B$ occurs in time $\ell$.
Then $B_{N} = B \cap T^{-N}(B)$ and $B_{N+1} = B \setminus B_{N}$,
and in particular $\partial B_\ell$ has zero probability.
Let
$$
B_{\ell,i} := B_\ell \cap W_i \setminus (W_1 \cup W_2 \cup \dots \cup W_{i-1}) \, ,  
\text{ for each $(\ell,i) \in \{N, N+1 \} \times \{1, \dots, k\}$.}
$$
Let $I$ be the set of pairs $(\ell,i)$ such that $B_{\ell,i} \neq \emptyset$.
Let also $J$ be the set of $(\ell,i,j)$ such that 
$(\ell,i) \in I$ and $0 \le j \le \ell-1$.
For each $\alpha = (\ell,i,j) \in J$, let $X_\alpha := T^j(B_{\ell,i})$.
%
Notice that $\{X_\alpha\}_{\alpha \in J}$ is a finite partition of $X$.
Moreover, each $\partial X_\alpha$ has zero probability, 
and so by \cref{l.freq} there exists an open set 
$V \supset \bigcup_{\alpha \in J} \partial X_\alpha$ and $n_* \in \N$ 
such that
\begin{equation}\label{e.freq2}
x \in X, \ n \ge n_* \quad \Rightarrow \quad 
\#   \{x,Tx,\dots,T^{n-1} x\} \cap V < \frac{\epsilon' n}{N+1} \, .
\end{equation}

For each $(\ell,i) \in I$, choose a point $y_{\ell,i} \in B_{\ell,i}$.
For each $j=0,1,\dots,\ell$, let $y_{\ell,i,j} := T^j(y_{\ell,i})$.
Applying the \cref{l.main},
we find $L_{\ell,i,0}$, \ldots, $L_{\ell,i,\ell-1}$
so that
\begin{align}
\|L_{\ell,i,j} -  A(y_{\ell,i,j}) \| &< \epsilon' \quad \forall j = 0, \dots, \ell -1, \quad\text{and} \label{e.main_norm} \\
\zeta \big( L_{\ell,i,\ell-1} \cdots L_{\ell,i,0} \big) &< \big( a Z(A) + \epsilon' \big) \ell, \label{e.main_zeta}
\end{align}
where $a = a_d \in (0,1)$ is a constant.

For each $\alpha = (\ell,i,j) \in J$,
let $\big\{ \tilde A_\alpha(x) \colon \E(x) \to \E(Tx)\big\}_{x\in X_\alpha}$
be the family of linear maps uniquely characterized by the following properties:
\begin{itemize}
	\item $\tilde A_\alpha(y_\alpha) = L_\alpha$;
	\item letting $m = m(i,j)$, $m'= m(i,j+1)$, 
	the local representation $\tilde A^{(m,m')}_\alpha(x)$ does not depend on $x\in X_\alpha$.
\end{itemize}

It follows from \eqref{e.main_norm} and \eqref{e.H_bounds} that
$$
\big\| \tilde A^{(m,m')}_\alpha(y_\alpha) - A^{(m,m')}_\alpha(y_\alpha) \big\| < C^2\epsilon'.
$$
So, by \eqref{e.unif_cont},
$$
\big\| \tilde A^{(m,m')}_\alpha(y_\alpha) - A^{(m,m')}_\alpha(x) \big\| < (C^2+1) \epsilon' \, ,
\quad \text{for all $x\in X_\alpha$.}
$$
It follows that
\begin{equation}\label{e.perturbation}
\| \tilde A_\alpha (x) - A(x) \| <  
\underbrace{C^2 (C^2+1) \epsilon'}_{< \epsilon \text{ (by \eqref{e.epsilon_2})}} \, ,
\quad \text{for all $x\in X_\alpha$.}
\end{equation}

For every 
$x \in B_{\ell, i}$,
the products $\tilde A_{\ell,i,\ell-1}(T^{\ell-1} x) \cdots \tilde A_{\ell,i,0}(x)$
and $L_{\ell,i,\ell-1} \cdots L_{\ell,i,0}$
have the same $(m(i,0), m(i,\ell))$-local representation.
It follows from \eqref{e.main_zeta} and \eqref{e.H_bounds} that
\begin{equation}\label{e.dirty}
x \in B_{\ell, i} \  \Rightarrow \ 
\zeta \Big( \tilde A_{\ell,i,\ell-1}(T^{\ell-1} x) \cdots \tilde A_{\ell,i,0}(x) \Big) 
< \big( a Z(A) + \epsilon' \big) \ell + 2C.
\end{equation}

\smallskip

Now consider the open cover $\{V\} \cup \{\interior X_\alpha\}_{\alpha \in J}$ of $X$.
Since $X$ is compact Hausdorff, 
we can find a continuous partition of unity $\{\psi\} \cup \{\phi_\alpha\}_{\alpha \in J}$
subordinate to this cover.
For each $x \in X$, define a linear map $\tilde A(x) \colon \E(x) \to \E(Tx)$ by
$$
\tilde A(x) := \psi(x) A(x) + \sum_{\alpha \in J} \phi_\alpha(x) \tilde A_\alpha (x) \, .
$$
By \eqref{e.perturbation}, we have $\|\tilde A (x) - A(x) \| < \epsilon$, 
and it follows from \eqref{e.epsilon_1} 
that $\tilde A(x)$ is invertible.
Thus $\tilde A \in \Aut(\E, T)$.
Also, $\zeta(\tilde{A}(x)) \le C_0$ for every $x$.

\smallskip

Take $n$ large enough so that
\begin{equation}\label{e.n}
n \ge n_* \quad \text{and} \quad 2 C_0 N < (\epsilon/3)n \, .
\end{equation}
We will give a uniform upper bound for $\zeta(\tilde{A}^n(x))$.
Fix $x \in X$ and write
$$
n = p + \ell_1 + \ell_2 + \cdots + \ell_r + q
$$
in such a way that 
the points
$$
x_1 = T^{p}(x), \  x_2 = T^{p+\ell_1}(x), \  \ldots, \  x_{r+1} = T^{p + \ell_1 + \cdots + \ell_r}(x)
$$
are exactly the points in the segment of orbit
$x$, $T(x)$, \ldots, $T^{n-1}(x)$ that belong to $B$.
Then $p$, $q \in [0,N]$ and $\ell_1$, \dots, $\ell_r \in [N,N+1]$.

The points $x_j$ such that $j \neq r+1$ and $\{x_j, Tx_j, \dots, T^{\ell_j-1} x_j \} \cap V = \emptyset$
will be called \emph{good}.
By subadditivity,
$$
\zeta \big( \tilde{A}^n(x) \big) 
\le \sum_{x_j \text{ is good}} \zeta \big(\tilde{A}^{\ell_j}(x_j)\big) 
+ C_0 \left( n - \sum_{x_j \text{ is good}} \ell_j \right)  \, .
$$
Notice the following estimates:
\begin{itemize}
\item If $x_j$ is good then $\zeta \big(\tilde{A}^{\ell_j}(x_j) \big)$ 
is less than the right hand side of \eqref{e.dirty} with $\ell=\ell_j$;
\item There are at most $r \le N^{-1} n$ good points;
\item By \eqref{e.freq2}, the number between large brackets is at most $2N+\epsilon'n$; 
equality may hold only in the case that each segment 
$\{x_j, Tx_j, \dots, T^{\ell_j-1} x_j \}$ (for $j=1,\dots,r$) 
contains at most one point of $V$.
\end{itemize}
Then we obtain:
$$
\zeta \big( \tilde{A}^n(x) \big) 
\le (a Z(A) + \epsilon') n + 2C N^{-1} n
+ C_0 (2N + \epsilon' n) \, .
$$
Using \eqref{e.epsilon_3}, \eqref{e.boring}, and \eqref{e.n},
we conclude that
$\zeta \big( \tilde{A}^n(x) \big) < (a Z(A) + \epsilon) n$.
So $Z(\tilde A) < a Z(A) + \epsilon$, as we wanted to prove.
\end{proof}

\begin{ack}
\Cref{l.moving spaces} and its proof were found (for a different purpose) together with Nicolas Gourmelon.
\end{ack}


\end{document}